\documentclass[]{article}
\usepackage{enumitem}
\usepackage{ mathrsfs }
\usepackage{geometry}
\usepackage[utf8]{inputenc}
\usepackage[english]{babel}
\usepackage{amsmath,amsthm,amsfonts,mathtools,mathrsfs}
\usepackage{setspace}
\usepackage{tikz}
\usetikzlibrary{shapes,arrows,backgrounds,positioning}
\usetikzlibrary{calc,arrows,decorations.markings,decorations.pathreplacing}
\usepackage{cleveref}
\crefname{equation}{}{}



\newcommand{\ve}{\varepsilon}
\newcommand{\R}{\mathbb{R}}
\newcommand{\N}{\mathbb{N}}

\newcommand{\eq}[1]{ 
\begin{equation}
\begin{split}
#1 
\end{split}
\end{equation}
}

\renewcommand{\H}{\mathcal{H}}
\renewcommand{\P}{\mathcal{P}}
\renewcommand{\l}{\langle}
\renewcommand{\r}{\rangle}

\newcommand{\parc}[1]{\frac{\partial}{\partial #1}} 
\newcommand{\parcs}[2]{\frac{\partial #1}{\partial #2}}

\newtheorem{remark}{Remark}[section] 
 
\newtheorem{definition}{Definition}[section] 
\newtheorem{theorem}{Theorem}[section] 
\newtheorem{lemma}{Lemma}[section]

\DeclareMathOperator{\Image}{Im}

\begin{document}

\title{Formal normal form of $A_k$ slow fast systems}
\author{H. Jardón-Kojakhmetov}

\maketitle
\abstract{An $A_k$ slow fast system is a particular type of singularly perturbed ODE. The corresponding slow manifold is defined by the critical points of a universal unfolding of an $A_k$ singularity. In this note we propose a formal normal form of $A_k$ slow fast systems.}

\section{Introduction}
In this note we propose a formal normal form of a particular class of slow fast systems. A slow fast system (SFS) is a singularly perturbed ODE usually written as
\eq{\label{sfs1}
\dot x &= f(x,z,\ve)\\
\ve\dot z &= g(x,z,\ve)
}
where $x\in\R^m$, $z\in\R^n$ and $0<\ve\ll 1$ is a small parameter, and where the over-dot denotes the derivative with respect to a time parameter $t$. Slow fast systems are often used as mathematical models of phenomena which occur in two time scales. Observe that as $\ve$ decreases, the time scale difference between $x$ and $z$ increases. A couple of classical examples of real life phenomena that were modeled by a SFS are the Zeeman's heartbeat and nerve-impulse models \cite{Zeeman1}. For $\ve\neq 0$, we can define a new time parameter $\tau$ by $t=\ve\tau$. With this new time $\tau$ we can write \cref{sfs1} as
\eq{\label{sfs2}
x' &=\ve f(x,z,\ve)\\
z' &= g(x,z,\ve),
}
where the prime denotes derivative with respect to $\tau$. An important geometric object in the study of SFSs is the \emph{slow manifold} which is defined by
\eq{
	S=\left\{ (x,z)\in\R^m\times\R^n \, | \, g(x,z,0)=0 \right\}.
}

When $\ve=0$, the manifold $S$ serves as the phase space of \cref{sfs1} and as the set of equilibrium points of \cref{sfs2}. In the rest of the document, we prefer to work with a SFS written as \cref{sfs2}. Furthermore, to avoid working with an $\ve$-parameter family of vector fields as in \cref{sfs2}, we plug-in into \cref{sfs2}  the trivial equation $\ve'=0$. To be more precise, we treat a $C^\infty$-smooth vector field defined as follows.

\begin{definition}[$A_k$ slow fast system]\label{def:eCDE} Let $k\in\mathbb N$ with $k\geq 2$. An $A_k$ slow fast system (for short $A_k$-SFS) is a vector field $X$ of the form
\eq{
X = \ve(1+\ve f_1)\parc{x_1}+\sum_{i=1}^{k-1}\ve^2f_i\parc{x_i}-\left(G_k-\ve f_k\right)\parc{z}+0\parc{\ve}.
}
where $G_k=z^k+\sum_{i=1}^{k-1}x_iz^{i-1}$ and where each $f_i=f_i(x_1,\ldots,x_{k-1},z,\ve)$ is a $C^\infty$-smooth function vanishing at the origin.
\end{definition}

\begin{remark} The slow manifold associated to an $A_k$-SFS is defined by
\eq{
S=\left\{ (x,z)\in\R^k \, |  \, z^k+\sum_{i=1}^{k-1}x_iz^{i-1}=0 \right\}.
}

The manifold $S$ can be regarded as the critical set of the universal unfolding of a smooth function with an $A_k$ singularity at the origin \cite{Arnold_singularities,Brocker}. Hence the name $A_k$-SFS.
\end{remark}

Observe that the origin is a non-hyperbolic equilibrium point of $X$ and thus, it is not possible to study its local dynamics with the classical Geometric Singular Perturbation Theory \cite{Fenichel}. In this case, a technique called blow-up \cite{DumRou1,DumRou2,Krupa1} is usually applied to desingularize the SFS. This methodology has been successfully used in many cases, e.g. \cite{BKK,Jardon3,Krupa2,Krupa20102841,Szmolyan2001419,vanGils}, where many of these deal with an $A_k$-SFS with fixed $k=2$ or $k=3$. Briefly speaking, the blow-up technique consists in an appropriate change of coordinates under which the induced vector field is regular or has simpler singularities  (hyperbolic or partially-hyperbolic). However, in this work we propose a normal form of $A_k$-SFS to be performed prior to the blow-up, see \cref{prop:nf}. This normalization greatly simplifies the local analysis of $A_k$-SFSs as shown in \cite{JardonThesis,Jardon3}.

\section{Formal normal form of an $A_k$-slow fast system }\label{sec:nf}

We regard the vector field $X$ of \cref{def:eCDE} as $X=F+P$, where $F$ and $P$ are smooth vector fields called ``the principal part'' and ``the perturbation'' respectively. That is
\eq{\label{eqF}
F =\ve\parc{x_1}+\sum_{i=2}^{k-1}0\parc{x_i}-G_k\parc{z}+0\parc{\ve}, \qquad\qquad\qquad
P = \sum_{i=1}^{k-1}\ve^2 f_i\parc{x_i}+\ve f_k\parc{z}+0\parc{\ve}.
}

The idea of the rest of the document is motivated by \cite{Sto10}. In short, we want to formally simplify the expression of $X$ by eliminating the perturbation $P$. The terminology used below is that of \cite{Sto10}.\smallskip

The vector field $F$ is quasihomogeneous  of type $r=(k,k-1,\ldots,1,2k-1)$ and quasidegree $k-1$ \cite{Arnold_singularities,Sto10}. From now on, we fix the type of quasihomogeneity $r$. A quasihomogeneous object of type $r$ will be called $r$-quasihomogeneous.

\begin{definition}[Good perturbation]\label{def:gp} Let $F$ be an $r$-quasihomogeneous vector field of quasidegree $k-1$. A good perturbation $X$ of $F$ is a smooth vector field $X=F+P$, where $P=P(x_1,\ldots,x_{k-1},z,\ve)$ satisfies the following conditions

\begin{itemize}
	\item $P$ is a smooth vector field of quasiorder greater than $k-1$, 
	\item $P=\sum_{i=1}^{k-1}P_i\parc{x_i}+P_k\parc{z}+0\parc{\ve}$, with $P|_{\ve=0}=0$.
\end{itemize}

\end{definition}

\paragraph{Notation}{ By $\P_{\delta}$ we denote the space of $r$-quasihomogeneous polynomials (in $k+1$ variables) of quasidegree $\delta$. By $\H_{\gamma}$ we denote the space of $r$-quasihomogeneous vector fields (in $\R^{k+1}$) of quasidegree $\gamma$ and such that for all $U\in\H_{\delta}$ we have $U=\sum_{i=1}^{k} U_k\parc{x_i}+0\parc{x_{k+1}}$. The formal series expansion of a function $f$ is be denoted by $\hat f$.
}

\begin{definition}[The inner product $\langle \cdot,\cdot\rangle_{r,\delta}$ \cite{Sto10}]\label{def:inner} Let $x=(x_1,\ldots,x_n)$, and $s,q\in\N^n$. Let $f,g\in\P_{\delta}$, that is
      \eq{
      f=\sum_{(r,s)=\delta} f_sx^s,
      }
      
      where $f_s\in\R$, $x^s=x_1^{s_1}\cdots x_n^{s_n}$; and similarly for $g$. Then the inner product  $\langle \cdot,\cdot\rangle_{r,\delta}$ is defined as
       \eq{
      \langle f, g\rangle_{r,\delta}=\sum_{(r,s)=\delta}f_sg_s\frac{(s!)^r}{\delta!},
      }
      
      where $(s!)^r=(s_1!)^{r_1}\cdots(s_n!)^{r_n}$, and where $(r,s)$ denotes the dot product $r\cdot s$. So for monomials one has
        \eq{
        \langle x^s, x^q\rangle_{r,\delta}=\begin{cases}
          \frac{(s_1!)^{r_1}\cdots(s_n!)^{r_n}}{\delta!} & \text{if} \quad s=q \quad \text{with}\quad (s,r)=\delta,\\
          0 & \text{otherwise.}
        \end{cases}
        }

      Accordingly, for vector fields: let $X=\sum_{i=1}^n X_i\parc{x_i}\in\H_\delta$, and $Y=\sum_{i=1}^n Y_i\parc{x_i}\in\H_\delta$. Then
      \eq{\label{def:innervf}
      \langle X,Y\rangle_{r,\delta}=\sum_{i=1}^n \langle X_i,Y_i \rangle_{r,\delta+r_i}.
      }

      \end{definition}

\begin{definition} [The operators $d$, $d^*$ and $\square$ \cite{Sto10}] The operator $d:\H_\gamma \to \H_{\gamma+k-1}$ (associated to $F$) is defined by $d(U)=[F,U]$ for any $U\in\H_{\gamma}$, where $[\cdot,\cdot]$ denotes the Lie bracket. The operator $d^*$ is the adjoint operator of $d$ with respect to the inner product of \cref{def:inner}. This is, given $U\in\H_\gamma$, $V\in\H_{\gamma+k-1}$ we have
\eq{
\langle d(U),V\rangle_{r,\gamma+k-1}=\langle U, d^*(V)\rangle_{r,\gamma}
}

For any quasidegree $\beta>k-1$, the self adjoint operator $\square_{\beta}:\H_\beta\to\H_\beta$ is defined by $\square_\beta(U)=dd^*(U)$ for all $U\in\H_\beta$.

\end{definition}

\begin{definition}[Resonant vector field \cite{Sto10}]\label{def:res} \leavevmode
\begin{itemize}
	\item We say that a vector field $U\in\H_\beta$ is resonant if $U\in\ker \square_\beta$.
	\item A formal vector field is called resonant if all its quasihomogeneous components are resonant.
\end{itemize}
\end{definition}

\begin{definition}[Normal Form \cite{Sto10}] A good perturbation $X=F+R$ of $F$ is a \emph{normal form with respect to $F$} if $R$ is resonant.
	
\end{definition}

It is important to note the following.

\begin{lemma}\label{lemma:ker} $\ker \square_{\beta}=\ker d^*|_{\H_\beta}$.

\end{lemma}

\begin{proof} Let $\alpha=k-1$, then $d:\H_{\gamma}\to\H_{\gamma+\alpha}$ and $d^*:\H_{\gamma+\alpha}\to\H_{\gamma}$. Due to the fact that $d^*$ is the adjoint of $d$, we have the decomposition $\H_\gamma=\Image d^*|_{\H_{\gamma+\alpha}}\oplus\ker d|_{\H_\gamma}$. Now let $U\in\H_{\gamma+\alpha}=\H_\beta$, then $\square_\beta(U)=dd^*(U)=0$ if and only if $d^*U\in\ker d$. Furthermore, $d^*U\in\Image d^*$. That is $d^*U\in \Image d^*\cap\ker d$. However $\Image d^*$ and $\ker d$ are orthogonal. Then $\square_{\beta}(U)=0$ if and only if $d^*U=0$.

\end{proof}

We now recall a result of \cite{Sto10} (Proposition 4.4), we only adapt it for the present context.

\begin{theorem}[Formal normal form \cite{Sto10}]\label{teo:Sto} Let $X=F+P$ be a good perturbation of $F$ as in definition \ref{def:gp}. Then there exists a formal diffeomorphism $\hat \Phi$ such that $\hat\Phi$ conjugates $\hat X$ to a vector field $F+R$, where $R$ is a resonant formal vector field in the sense of definition \ref{def:res}.

\end{theorem}

Finally, we present our result. In short, we prove that the resonant vector field $R$ in \cref{teo:Sto} associated to $F$ given by \cref{eqF} is $R=0$. 

\begin{theorem}\label{prop:nf} Let $X=F+P$ be a good perturbation of the vector field
\eq{
F=\ve\parc{x_1}+\sum_{i=2}^{k-1}0\parc{x_i}-\left( z^k+\sum_{j=1}^{k-1}x_jz^{j-1} \right)\parc{z}+0\parc{\ve}.
}
Then, there exists a formal diffeomorphism $\hat\Phi$ that conjugates $\hat X$ with $F$, this is $\hat\Phi_*\hat X=F$.
\end{theorem}

\begin{proof}
From theorem \ref{teo:Sto} and lemma \ref{lemma:ker} we will show that if $P\in\ker d^*|_{\H_{\geq k}}$ then $P=0$. Let us start by rewriting $d^*(P)$ in a more workable format. 

To simplify the notation, let $\alpha\geq k$, $P\in\H_\alpha$, $\beta=\alpha-k+1$, and $Q\in\H_{\beta}$; and let $x=(x_1,\ldots,x_{k-1},z,\ve)=(x_1,\ldots,x_{k-1},x_k,x_{k+1})$. If $D$ is an operator, its adjoint with respect to the inner product \cref{def:inner} is always denoted as $D^*$.

We start with the inner product (\cref{def:inner})
\eq{
\l d(Q),P\r_{r,\alpha}=\l Q,d^*(P)\r_{r,\beta}.
}

We can write $d(Q)=\sum_{i=1}^{k+1} F(Q_i)-Q(F_i)$, where $F(Q_i)=\sum_{j=1}^{k+1} F_j\parcs{Q_i}{x_j}$ and similarly for $Q(F_i)$, then
\eq{\label{dd1}
\l d(Q),P \r_{r,\alpha} &= \sum_{i=1}^{k+1} \l F(Q_i)-Q(F_i), P_i \r_{r,\beta} = \sum_{i=1}^{k+1} \l F(Q_i),P_i\r_{r,\alpha+r_i} - \l Q(F_i),P_i \r_{r,\alpha+r_i} \\
							   &= \sum_{i=1}^{k+1} \l Q_i, F^*(P_i)\r_{r,\beta+r_i}   - \l Q(F_i),P_i \r_{\alpha+r_i} = \sum_{i=1}^{k+1} \l Q_i, F^*(P_i)\r_{r,\beta+r_i}   - \sum_{j=1}^{k+1} \l Q_j,\left( \parcs{F_i}{x_j} \right)^*(P_i) \r_{\beta+r_j}\\
							   &= \sum_{i=1}^{k+1} \l Q_i, F^*(P_i)                    - \sum_{j=1}^{k+1} \left( \parcs{F_j}{x_i} \right)^*(P_j) \r_{\beta+r_i} 
}
Comparing \cref{dd1} to $\l Q, d^*(P)\r_{r,\beta}$ we can write 
\eq{\label{eq:d}
d^*(P)=\begin{bmatrix}
	F^*-\left( \parcs{F_1}{x_1 } \right)^* & -\left( \parcs{F_2}{x_1 } \right)^* & \cdots & -\left( \parcs{F_{k+1}}{x_1 } \right)^*\\[3ex]
	-\left( \parcs{F_1}{x_2 } \right)^* & F^*-\left( \parcs{F_2}{x_2 } \right)^* & \cdots & -\left( \parcs{F_{k+1}}{x_2 } \right)^*\\
	\vdots & \vdots & \ddots & \vdots\\
	-\left( \parcs{F_1}{x_{k+1} } \right)^* & -\left( \parcs{F_2}{x_{k+1} } \right)^* & \cdots & F^*-\left( \parcs{F_{k+1}}{x_{k+1} } \right)^*
\end{bmatrix}
	\begin{bmatrix}
		P_1\\
		P_2\\
		\vdots\\
		P_{k+1}
	\end{bmatrix}.
}

Plugging in the expressions of $F$ and $P$ into \eqref{eq:d} we get
\eq{\label{eq:d2}
d^*(P)=\begin{bmatrix}
	F^* & 0 & \cdots & 0 & 1 & 0\\
	0 & F^* & \cdots & 0 & z^* & 0 \\
	\vdots & \vdots & \ddots & \vdots & \vdots \\
	0 & 0 & \cdots & F^* & \left(z^{k-1}\right)^* & 0 \\
	0 & 0 & \cdots & 0 & F^*+Z^* &0\\
	-1 & 0 & \cdots & 0 & 0 & F^*
\end{bmatrix}
	\begin{bmatrix}
		P_1\\
		P_2\\
		\vdots\\
		P_{k-1}\\
		P_k\\
		0
	\end{bmatrix}=0.
}
where $Z^*=\left(kz^{k-1}+\sum_{i=2}^{k-1}(i-1)x_iz^{i-2}\right)^*$. Now note that \eqref{eq:d2} implies  $F^*(P_j)=0$ for all $j=2,\ldots,k-1$ and $P_1=P_k=0$.
\begin{remark} For $k=2$, the result is trivial: we have $F=\ve\parc{x_1}-(z^2+x_1)\parc{z}+0\parc{\ve}$, and therefore $d^*(P)=0$ is written as
	\eq{
	d^*(P)=\begin{bmatrix}
		F^* & 1 & 0 \\
		0 & F^*+2z^* & 0\\
		-1 & 0 & F^*
	\end{bmatrix}
	\begin{bmatrix}
		P_1\\
		P_2\\
		0
	\end{bmatrix}=0,
	}
	which immediately implies $P_1=P_2=0$.
\end{remark}

Now, we study $F^*(P_j)=0$. Recall that $P=P(x_1,\ldots,x_{k-1},z,\ve)$ is not any vector field, but it has the property that $P(x_1,\ldots,x_{k-1},z,0)=0$. That is, we can write
\eq{
P=\sum_{i=1}^{k-1}\ve \bar P_i\parc{x_i}+\ve\bar P_k\parc{z}+0\parc{\ve},
}
where $\bar P_j\in\P_{\alpha+r_j-2k+1}$. This is because the (quasihomogeneous) weight of $\ve$ is $2k-1$. Now, since it is complicated to work with the adjoint, we first rewrite the problem $F^*(\ve\bar P_j)=0$. We then prove that  $F^*(\ve\bar P_j)=0$ implies that $\bar P_j=0$. 

Note that $F^*(\ve\bar P_j)=0$ is equivalent to $\l Q, F^*(\ve\bar P_j)\r_{\alpha+r_j-k+1}=0 $ for all $Q\in\P_{\beta+r_j}$. Next, we use the definition of $F^*$ that is 
\eq{\label{auu0}\l Q, F^*(\ve\bar P_j)\r_{r,\beta+r_j}=\l F(Q),\ve\bar P_j\r_{r,\alpha+r_j}=0.}

We will now show that if $\l F(Q),\ve\bar P_j\r_{r,\alpha+r_j}=0$ for all $Q\in\P_{\beta+r_j}$, then $\bar P_j=0$. Note that by \cref{auu0}, this is the same as proving that $F^*(\ve\bar P_j)=0$ implies $\bar P_j=0$.

Start by choosing an element $x^q$ of the basis of $\P_{\beta+r_j}$, this is
\eq{
x^q=x_1^{q_1}\cdots x_{k-1}^{q_{k-1}}z^{q_k}\ve^{q_{k+1}}, \qquad (r,q)=\beta+r_j.
}

Then we have
\eq{
F(x^q)=q_1x_1^{q_1-1}\cdots x_{k-1}^{q_{k-1}}z^{q_k}\ve^{q_{k+1}+1}-\left( z^k +\sum_{i=1}^{k-1}x_iz^{i-1} \right)q_k x_1^{q_1}\cdots x_{k-1}^{q_{k-1}}z^{q_k-1}\ve^{q_{k+1}}.
}

Let us write $\ve\bar P_j\in\P_{\alpha+r_j}$ as
\eq{
\ve \bar P_j =\ve\sum_{(r,p)=\alpha+r_j-2k+1} a_p x_1^{p_1}\cdots x_{k-1}^{p_{k-1}}z^{p_k}\ve^{p_{k+1}},
}
where $a_p\in\R$. We now proceed by recursion on the exponent of $\ve$. Let $q_{k+1}=0$, then the inner product $\l F(Q),\ve\bar P_j\r_{\alpha+r_j}$ has only one term since $F(Q)$ has only one monomial containing $\ve$. That is
\eq{\label{eq:aux1}
\l F(Q),\ve\bar P_j\r_{\alpha+r_j}|_{q_{k+1}=0}=\l q_1x_1^{q_1-1}\cdots x_{k-1}^{q_{k-1}}z^{q_k}\ve^{}, \ve a_p x_1^{p_1}\cdots x_{k-1}^{p_{k-1}} z^{p_k}\r_{r,\alpha+r_j}=0.
}

We naturally consider $q_1>0$. If $q_1=0$, then the equality is automatically satisfied. Recalling the definition \ref{def:inner} of the inner product, the  equality \eqref{eq:aux1} means that
\eq{
\l q_1x_1^{q_1-1}\cdots x_{k-1}^{q_{k-1}}z^{q_k}\ve, \ve a_p x_1^{p_1}\cdots x_{k-1}^{p_{k-1}} z^{p_k}\r_{r,\alpha+r_j}=q_1 a_p \frac{(q!)^r}{(\alpha+r_j)!}=0,
}
and therefore from \eqref{eq:aux1} we have
\eq{\label{eq:coef1}
a_p=a_{q_1-1,p_2,\ldots,p_k,1}=0,  
}
for all  $q_1>0, \, p_2,\ldots,p_k\geq 0$ (naturally, also satisfying the degree condition $(r,p)=\alpha+r_j$). Next, let $q_{k+1}=1$. Then
\eq{
F(x^q)=q_1x_1^{q_1-1}\cdots x_{k-1}^{q_{k-1}}z^{q_k}\ve^{2}-\left( z^k +\sum_{i=1}^{k-1}x_iz^{i-1} \right)q_k x_1^{q_1}\cdots x_{k-1}^{q_{k-1}}z^{q_k-1}\ve^{}.
}

Once again, the inner product $\l F(Q),\ve\bar P_j\r_{r,\alpha+r_j}$ has only one term, now this is due to the fact that all coefficients $a_p$ of monomials containing $\ve$ are zero due to \eqref{eq:coef1}. Then
\eq{
\l F(Q),\ve\bar P_j\r_{\alpha+r_j}|_{q_{k+1}=1}=\l q_1x_1^{q_1-1}\cdots x_{k-1}^{q_{k-1}}z^{q_k}\ve^{2}, \ve a_p x_1^{p_1}\cdots x_{k-1}^{p_{k-1}} z^{p_k}\ve\r_{r,\alpha+r_j}=0.
}

Therefore, similarly as above, we have the condition
\eq{\label{eq:coef2}
a_p=a_{q_1-1,p_2,\ldots,p_k,2}=0,  
}
for all  $q_1>0, \, p_2,\ldots,p_k\geq 0$ (naturally, also satisfying the degree condition $(r,p)=\alpha+r_j$). By recursion arguments, assume $q_{k+1}=n$ and that all the coefficients
\eq{
a_p=a_{p_1,p_2,\ldots,p_k,m}=0, \qquad \forall m\leq n.
}

Then again the inner product $\l F(Q),\ve\bar P_j\r_{r,\alpha+r_j}$ has only one term, namely
\eq{
\l F(Q),\ve\bar P_j\r_{\alpha+r_j}|_{q_{k+1}=n}=\l q_1x_1^{q_1-1}\cdots x_{k-1}^{q_{k-1}}z^{q_k}\ve^{n+1}, \ve a_p x_1^{p_1}\cdots x_{k-1}^{p_{k-1}} z^{p_k}\ve^n\r_{r,\alpha+r_j}=0.
}

The latter then implies
\eq{
a_p=a_{q_1-1,p_2,\ldots,p_k,n+1}=0.
}

This finishes the proof of $\l F(Q),\ve\bar P_j\r_{r,\alpha+r_j}=0$ implies $\bar P_j=0$. 
\end{proof}

\begin{remark} \Cref{prop:nf} together with Borel's lemma \cite{Brocker}, imply that an $A_k$-SFS $X=F+P$ is \emph{smoothly} conjugate to a smooth vector field $Y=F+H$ where $H$ is flat at the origin. The benefits of this normal form are exploited in \cite{JardonThesis,Jardon3}.
\end{remark}

\section*{Acknowledgments} The author gratefully acknowledges Henk Broer, Robert Roussarie, and Laurent Stolovitch for fruitful discussions and valuable comments and suggestions. This work is partially supported by a CONACyT postgraduate grant.

\bibliographystyle{plain}
\bibliography{phdbib}

\end{document}